\documentclass[11pt]{article}
\usepackage[margin=1in]{geometry}    
\usepackage{amssymb,amsmath}
\usepackage{amstext,amsthm, amssymb,amsfonts}
\usepackage{multirow}
\usepackage{amsthm}
\usepackage{graphicx}
\usepackage{url}

\newtheorem{theorem}{Theorem}[section]

\title{Swap connectivity for two graph spaces between simple and pseudo graphs and disconnectivity for triangle constraints}

\date{}        
\begin{document}
\maketitle

\centerline{\scshape Joel Nishimura}
\medskip
{\footnotesize
 \centerline{School of Mathematical and Natural Sciences}
   \centerline{ Arizona State University, Glendale, AZ 85306-4908, USA}
} 

\begin{abstract}
With sufficient time, double edge-swap Markov chain Monte Carlo (MCMC) methods are able to sample uniformly at random from many different and important graph spaces.  For instance, for a fixed degree sequence, MCMC methods can sample any graph from: simple graphs; multigraphs (which may have multiedges); and pseudographs (which may have multiedges and/or multiple self-loops). In this note we extend these MCMC methods to `multiloop-graphs', which allow multiple self-loops but not multiedges and `loopy-multigraphs' which allow multiedges and single self-loops.  We demonstrate that there are degree sequences on which the standard MCMC methods cannot uniformly sample multiloop-graphs, and exactly characterize which degree sequences can and cannot be so sampled.  In contrast, we prove that such MCMC methods can sample all loopy-multigraphs.  Taken together with recent work on graphs which allow single self-loops but no multiedges, this work completes the study of the connectivity (irreducibility) of double edge-swap Markov chains for all combinations of allowing self-loops, multiple self-loops and/or multiedges.  Looking toward other possible directions to extend edge swap sampling techniques, we produce examples of degree and triangle constraints which have disconnected spaces for  all edges swaps on less than or equal to $8$ edges.     
\end{abstract}

%
%

\section{Introduction}

Consider graphs $G=(V,E)$, with vertex set $V$ of size $n$ and edges $E$ of the form $(u,v)\in V \times V$.  Let multiloop-graphs be those where multiple self-loops (edges of the form $(u,u)$) are allowed, but all other edges are required to be simple. Similarly, let loopy-multigraphs be those where self-loops are required to be simple, but other edges are allowed to be multiedges\footnote{In the notation of Eggleton and Holton, multiloop-graphs would be of type $(0,1,\infty)$, loopy-multigraphs would be of type $(0,\infty,1)$, while simple graphs and pseudographs are of types $(0,1,1)$ and $(0,\infty,\infty)$ respectively.}.  For a vertex $u$, we denote the multiset of adjacent vertices, or neighbors, as $N(u)$, and the degree sequence as $k_u = |N(u)|$. To simplify our analysis, we will assume that $\{k_u \}$ has no nodes with $0$ degree. 

Frequently referred to as a `configuration model', the ability to sample graphs with a fixed degree sequence uniformly at random underlies a number of hypothesis tests in network science \cite{MRCconfiguration}.  The most well known such configuration model uses the straightforward and simple `stub-matching' procedure to sample pseudographs.  However, when sampling other graphs spaces, which may or may not have multiedges, self-loops or multiple self-loops, it can be difficult to sample graphs uniformly.  Indeed, in many graphs spaces there are not ways to directly sample degree constrained graphs, and instead Markov chain Monte Carlo techniques are used, the most common of which is based on double edge-swaps.  

Double edge-swaps date back to Petersen in 1891 \cite{petersen1891theorie}, and are a well established technique of transforming one graph into a closely related graph with the same degree sequence.  Let $(u,v),(x,y)\leadsto (u,x),(v,y)$ denote the double edge swap that replaces edges $(u,v)$ and $(x,y)$ with $(u,x)$ and $(v,y)$.  Notice that such a double edge swap creates a new graph with the same degree sequence, but may or may not create/destroy self loops and/or multiedges. For a degree sequence $\{ k_u\}$ consider the graph of graphs (gog) $\mathcal{G}(\{k_u\})=(\mathcal{V},\mathcal{E})$ where each element $G\in\mathcal{V}$ is a graph, and edges $(G_i, G_j)\in\mathcal{E}$ correspond to a double edge swap that transforms $G_i$ to $G_j$.  

Random walks on a gog are thus random walks across possible degree constrained graph structures, and if the appropriate transition probabilities are used \cite{MRCconfiguration,carstens2016switching}, then a gog forms the backbone of a MCMC sampler.  However, such an MCMC sampler (along with a family of related MCMC methods \cite{carstens2015Proof,carstens2016curveball}) can only sample from all possible graphs if the underlying gog is connected.  Indeed, previous work has shown gog connectivity for the graph of simple graphs \cite{zhang2010traversability, bienstock1994degree, berge1962theory, eggleton1981simple, taylor1981constrained}, the graph of pseudographs  \cite{eggleton1979graph}, and the graph of multigraphs \cite{hakimi1963realizability}.  However, the graph of loopy-graphs is not connected for all degree sequences, though there are techniques for determining exactly which degree sequences are disconnected \cite{nishimura2017uniformly}. In this paper we examine the last two natural graph spaces between simple and pseudographs: multiloop-graphs, which allow multiple self-loops but not multiedges; and loopy-multigraphs, which allow multiedges and single self-loops.  Let $\mathcal{G}_{ll}$ and $\mathcal{G}_{lm}$ be subgraphs of $\mathcal{G}$ restricted to be those where the elements of $\mathcal{V}$ are multiloop-graphs and loopy-multigraphs respectively.

We find that $\mathcal{G}_{lm}$ are connected for all degree sequences, implying that standard MCMC sampling techniques can sample from this space.  In contrast, we show that $\mathcal{G}_{ll}$ are disconnected for some degree sequences, and develop the exact criterion to determine for which degree sequences this is so. This work fills in the remaining potential graph spaces which may or may not allow multiedges, single self-loops, or multiple self-loops, as shown in table \ref{tableOne}.  

\begin{table}\begin{centering}
\resizebox{\textwidth}{!}{
\begin{tabular}{ l | c c c c c c }
  & pseudo- & \bf{\emph{ loopy-}} &\bf{\emph{  multiloop-}} & loopy-& multigraph & simple \\
 &graph&\bf{\emph{ multigraph}}&\bf{\emph{graph}}&graph &&\\
  \hline			
  self-loops: & m & \bf{\emph{ s}} & \bf{\emph{ m}} & s&0 &0 \\
  edges: & m & \bf{\emph{ m}} & \bf{\emph{ s}}& s& m&s \\
  \hline
  gog connected & y & \bf{\emph{ y}} & \bf{\emph{ n}} & n& y& y\\
  \hline  
\end{tabular}}
\caption{Between pseudographs and simple graphs lie a range of possible graphs spaces depending on whether multiple self-loops are allowed (m), only single self-loops are (s) or no self-loops are (0) and whether edges are allowed to be multiedges (m) or not (s).  Some spaces have a connected graph of graphs for all degree sequences while others do not.  This paper establishes the emphasized columns. \label{tableOne} }
\end{centering}
\end{table}

One benefit of MCMC approaches is that they can be easily extended to new and varied constraints.  For instance, researchers have explored using triple edge-swaps and $k$ edge-swaps  to sample from graphs with constrained degree sequences and a fixed number of triangles \cite{tabourier2011generating}. Indeed, conserving the number of triangles would provide a useful null model for social network researchers interested in understanding clustering. However, we present examples of degree sequences and triangle sequences such that no edge-swaps on less than $8$ edges can connect the space.  Indeed, though small, these examples involve graphs relatively distant from each, and thus may raise concerns even for methods which do not strictly enforce triangle constraints, but instead attempt to bias swaps towards a triangle based constraint \cite{colomer2014double}.

\section{Loopy-Multigraphs}

To see that $\mathcal{G}_{lm}$ is connected, one need only lean heavily on the already established fact that for all degree sequences the graph of multigraphs is connected \cite{hakimi1963realizability}.  

\begin{theorem}
For all degree sequences $\{k_u\}$, $\mathcal{G}_{lm}(\{k_u\})$ is connected.
\end{theorem}

\begin{proof}

We will show that there exists a path between any two loopy-multigraphs $G_i,G_j\in \mathcal{V}$.  To construct this path, we will first show that $G_i$ and $G_j$ are connected to graphs $G_i^*$ and $G_j^*$, which either both have only a single self loop at the same node, or don't have any self-loops at all.  The connectivity of graphs of multigraphs \cite{hakimi1963realizability} then guarantees that $G_i^*$ is connected to $G_j^*$.  

For any loopy-multigraph with two self-loops $(u,u)$ and $(v,v)$ double edge swap  $(u,u),(v,v)\leadsto(u,v)(u,v)$ produces a new loopy-multigraph with two less self-loops.  Thus, both $G_i$ and $G_j$ are connected to graphs $G_i'$ and $G_j'$ which contain at most a single self-loop.  Now consider the following three cases:
\begin{enumerate}
\item $G_i'$ and $G_j'$ have no self-loops: thus, $G_i^*=G_i'$ and $G_j^* = G_j'$.
\item Both $G_i'$ and $G_j'$ have a single self-loop at the same vertex $u$: thus, $G_i^*=G_i'$ and $G_j^* = G_j'$.
\item WLOG $G_i'$ has a self-loop at vertex $u$ but $G_j'$ does not. We will produce a new graph connected to $G_i'$ but without a self-loop at $u$. This argument, or potentially two applications of this argument, lead to the first case. 

If $(u,u)\in V_i'$ but $(u,u)\not \in V_j'$ then the sum of the degrees in the subgraph of $G_i'$ on vertices $V\setminus \{u\}$ is two greater than that of the subgraph of $V_j'$ on the same vertices.  Thus, the sum of degrees in the subgraph of $G_i'$ on vertices $V\setminus  \{u\}$ is at least $2$, and there must exist edge $(v,w)\in E_i'$ with $u\not = v,w$.  Swap $(u,u),(v,w)\leadsto(u,v),(u,w)$ produces a new loopy-multigraph without any self-loops.  
\end{enumerate} 
As mentioned, connectivity from $G_i^*$ to $G_j^*$ thus follows from the connectivity of graphs of multigraphs \cite{hakimi1963realizability}

\end{proof}

\section{Multiloop-graphs}

Here we determine the exact criterion for which degree sequences $\mathcal{G}_{ll}(\{k_u\})$ is connected. 

\begin{theorem}
For  $\mathcal{G}_{lm}(\{k_u\})$ with $|\mathcal{V}| \ge 2$ then $\mathcal{G}_{lm}(\{k_u\})$ is connected if and only if there exists both: 
\begin{enumerate}
\item a vertex $u$ such that $k_u$ is odd;
\item and a vertex $v$ such that $k_v-(n-1)$ is negative or odd\footnote{Recall the assumption that $\{k_u\}$ has no vertices with degree $0$.}.
\end{enumerate}
\end{theorem}

\begin{proof}
These conditions are necessary for the following two reasons.  First, if there is no vertex with an odd degree, then the graph composed of $\frac{k_w}{2}$ self-loops at  each vertex is a valid graph---but swapping two self-loops on different vertices would create a multiedge, and is thus not a valid swap in $G_{ll}$.  Meanwhile, a degree sequence that lacks the second property, can wire the complete graph with possibly some number of self-loops--but any double edge swap on such a graph would also make a multiedge.  

To demonstrate that these conditions are sufficient we will show that any two graphs $G_s$ and $G_t$ in  $\mathcal{G}_{ll}$ are each connected to a graph with $\lfloor \frac{k_u}{2}\rfloor$ self-loops at every vertex---all other edges necessarily form a simple graph with degree distribution $\{k_i \pmod 2\}$. Connectivity on simple graphs will then give that there is a path from $G_s$ to $G_t$. Thus, it only remains to be shown that any graph is connected to a graph with $\lfloor \frac{k_u}{2}\rfloor$ self-loops at each vertex. 

Of the graphs connected to a graph $G$ let $G_l$ be the graph with the maximum number of self-loops.  Since rewiring an open wedge $(u,v)$, $(u,x)$ creates a self-loop, $G_l$ cannot contain any open wedges. Thus every subset of three vertices in $G_l$ contains either an isolated edge, or a triangle.  Let $K$, represent the largest complete graph in $G_l$. Since there exists $v$ where $k_v-(n-1)$ is negative or odd then there is at least one vertex $u$ with positive degree not in $K$. Since $K$ is maximal, at least some edge between $K$ and $u$ is absent, and since $G_l$ contains no wedges, $u$ must have no neighbors in $K$.  Since $u$ has positive degree there exists edge $(u,w)$ (possibly with $u=w$). By a similar argument as above, $w$ is also independent of $K$.  

Let $c = |K|$. Suppose to the contrary that $c\ge4$, with vertices $x$, $y$, $a$ and $b$ in $K$.  Notice that because edges $(u,x)$ and $(w,y)$ are absent the double edge swap on $(u,w)$, $(x,y)$ is valid, but doing so creates wedges $xay$ and $bxu$ which can be used to create two self-loops, contradicting that $G_l$ is maximal. Thus $G_l$ must be composed of triangles, isolated edges and self-loops. 

Since there exists at least one vertex $p$ with odd degree, there must be an isolated edge $(p,q)$.  However, if $G_l$ contains a triangle on vertices $x$, $y$ and $z$, then valid rewire $(p,q)$ and $(x,y)$ creates open wedges $pxz$ and $qyz$, and similarly contradicts the maximality of $G_l$. Thus $G_l$ is composed of only isolated edges and self-loops, and so must have $\lfloor \frac{k_i}{2}\rfloor$ self-loops at every vertex and with all other edges comprising a simple graph with degree distribution $\{k_i \pmod 2\}$.  
\end{proof}

While it is thus the case that $\mathcal{G}_{ll}$ is not connected for all degree sequences, it is worth noting that the two properties required for it to be connected are likely true of almost all empirical degree sequences investigated.  Indeed, all sparse networks with a  single leaf would have a connected $\mathcal{G}_{ll}$.  Next, we consider an extension of MCMC techniques to a space where edge-swaps may have significant limitations. 

\section{Additional constraints}

\begin{figure}\begin{centering}
 \parbox[c]{.5\linewidth}{\includegraphics[width=1\linewidth]{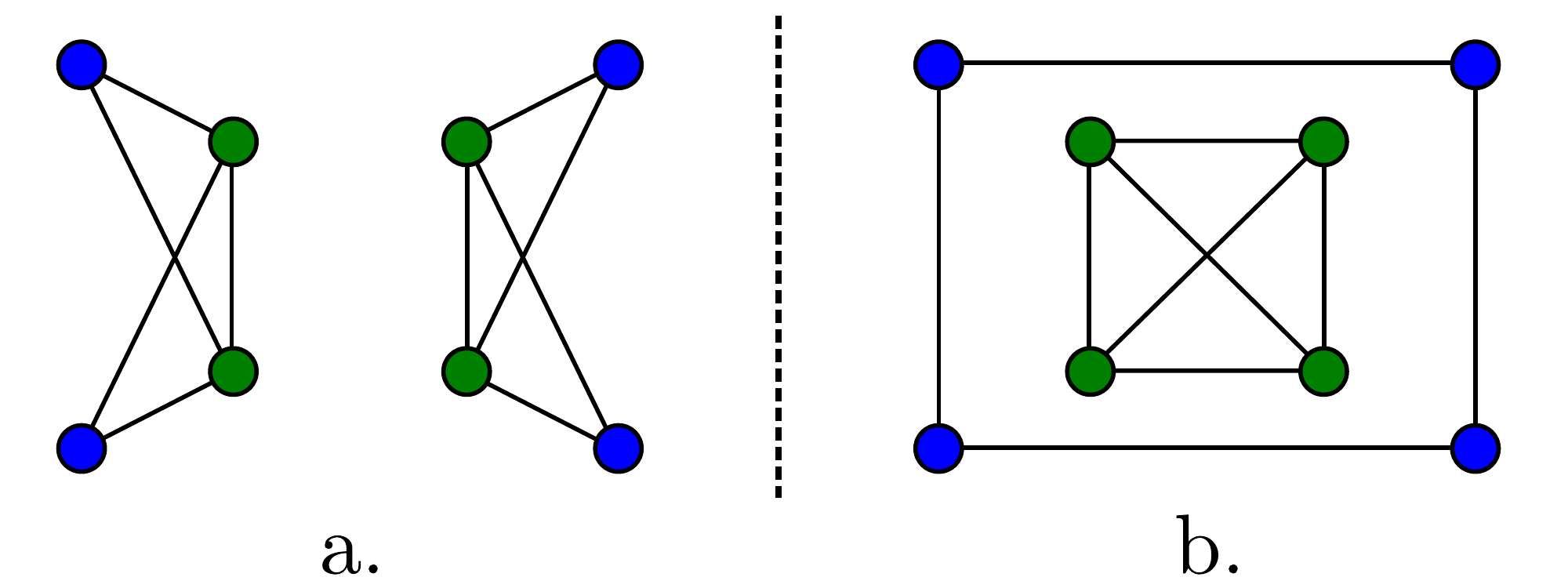}  } 
\begin{tabular}{|c |c|} 
\hline
number of  & number of graphs \\
 		triangles & with given deg. seq. \\ \hline
0 & 2052 \\
1 & 2664 \\
2 & 1152 \\
3 & 168 \\
4 & 21\\ \hline
\end{tabular}
\caption{Of all possible configurations with exactly 4 triangles, the degree sequence $\{3,3,3,3,2,2,2,2 \}$ has exactly two isomorphism classes. Taken together, the $18$ graphs in isomorphism class (a) and the $3$ graphs in isomorphism class (b) account for all $21$ graphs with $4$ triangles enumerated in an exhaustive search of the degree sequence (right).  \label{triImp}}
\end{centering}
\end{figure}

\begin{figure}\begin{centering}
 \parbox[c]{.5\linewidth}{\includegraphics[width=1\linewidth]{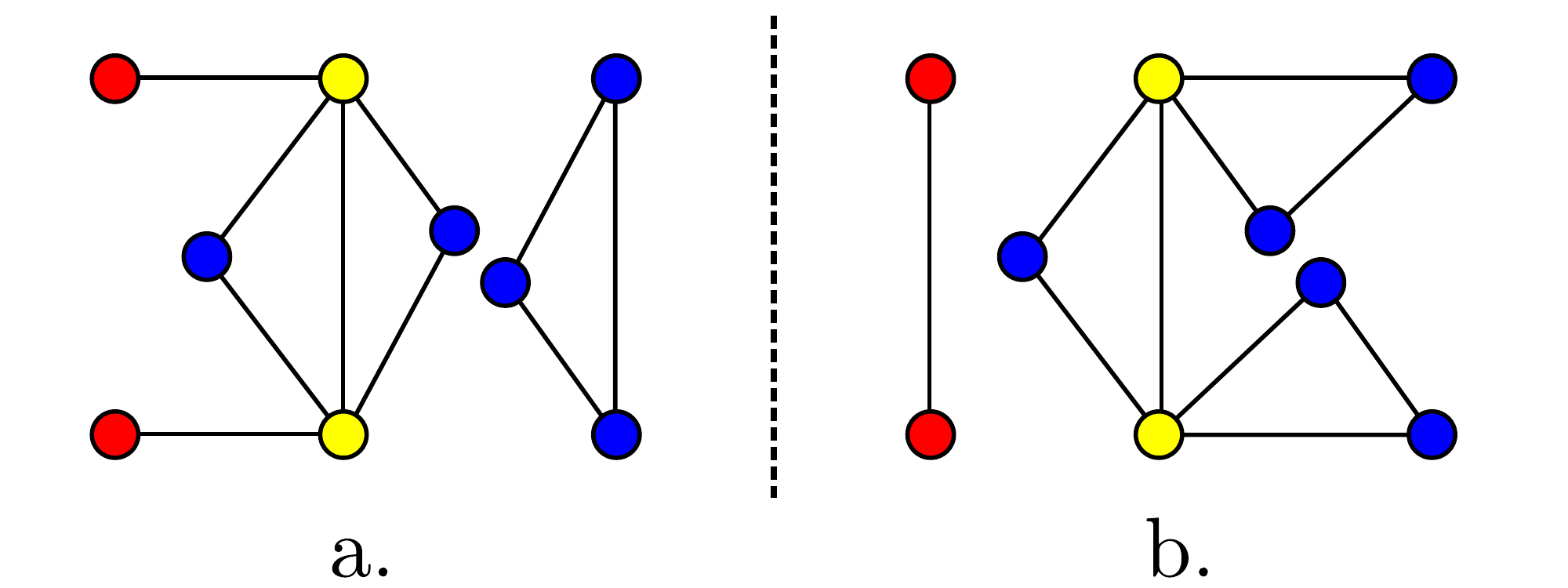} }
 \begin{tabular}{|c |c|}
 \hline
graphs w/ deg. seq. $\{k_i\}$& 5075 \\
subset w/ triangles $\{t_i\}$ & 50 \\[2mm]
size of class (a) & $  20 $ \\[3mm]
size of class (b) & $ 30$ \\[2mm]
 \hline
\end{tabular}
\caption{Of all possible simple graphs with the degree sequence $\{d_i\} = \{4,4,2,2,2,2,2,1,1\}$ and corresponding triangle sequence $\{t_i\} = \{2,2,1,1,1,1,1,0,0 \}$  there are exactly two isomorphism classes. \label{triImpSeq}}
\end{centering}
\end{figure}

\begin{figure}\begin{centering}
\includegraphics[width=1\linewidth]{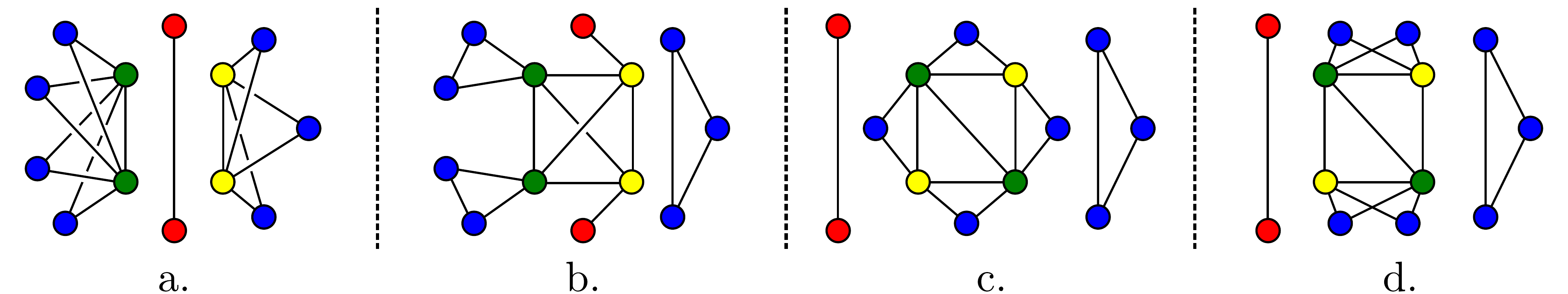}
\caption{There are $302,914,500$ graphs with degree sequence $\{5,5,4,4,2,2,2,2,2,2,2,1,1\}$ and of these $1,715$ have triangle sequence $\{4,4,3,3,1,1,1,1,1,1,1,0,0\}$ and these belong to four isomorphism classes of sizes $35$, $420$, $840$ and $420$. \label{triImpSeq2}}
\end{centering}
\end{figure}

In a number of situations it is natural to consider sampling the  subset of simple graphs with the additional constraint that the number of triangles in the graph remains constant. An even more restrictive constraint may requires that graphs preserve a triangle sequence $\{t_u\}$, where each vertex $u$ is a member of exactly $t_u$ triangles.   Notice that if the degree sequence is fixed, these two constraints are exactly those that preserve the clustering coefficient and the local clustering coefficient respectively.  

Previous studies have examined the ability of $k$ edge-swaps (a generalization of double edge swaps) to sample from graphs with a fixed number of triangles \cite{tabourier2011generating}.  A $k$ edge-swap involves selecting $k$ edges and cyclically permuting one endpoint in each edge. For example, a $3$ edge-swap on $(u,v)$, $(x,y)$ and $(w,z)$ might create edges $(u,y)$, $(x,z)$ and $(w,v)$.  However, while there is numerical evidence that for some degree sequences $3$ edge-swaps and $4$ edge-swaps are effective at sampling simple graphs with a constrained number of triangles, we show that there are degree sequences for which these spaces are not connected.

Consider, as in figure \ref{triImp}, degree sequence $\{3,3,3,3,2,2,2,2 \}$ with $4$ total triangles. As checked exhaustively, this degree sequence has exactly $2$ isomorphism classes with $4$ triangles and thus it then immediately follows that the associated graph of graphs is disconnected for $k$ edge-swaps for $k<8$ since there are $8$ different edges between the two isomorphism classes.  

To see that the most common family of swaps cannot also preserve triangle sequences, consider, figures \ref{triImpSeq} and \ref{triImpSeq2}, which display the the only simple graph isomorphism classes with the given degree and triangle sequences. Again, notice, that since the second such isomorphism class differs from all other isomorphism classes by at least $8$ edges and thus are disconnected under the most commonly used edge swaps.

The variety between the $3$ discussed counter examples suggests that the problems of constraining triangle counts and triangle sequencing may be more general than the previously mentioned problems associated with including self-loops. Indeed, relative to the size of these graphs, the various isomorphism classes discussed differ by many edges, and thus likely also pose a problem to methods which might subsample graphs which meet triangle constraints from a gog without triangle constraints, even with biased transition weights.

\section{Conclusion}

While it may have been hoped that double edge-swap MCMC techniques would be able to sample all graphs for all degree sequences from any of the four graph space between simple graphs and pseudographs, this is not the case.  Instead, when self-loops are permitted, but multiedges are not, there are families of degree sequences for which double edge-swap MCMC sampling will not sample every graph.  However, it is likely the case that the problematic degree sequences are extremely rare in empirical settings. In any case, testing for whether the degree sequence is problematic is not too computationally expensive.  

In contrast, the negative examples of MCMC sampling on simple graphs with a constrained degree sequence and triangle count, as well as the examples for a constrained degree sequence and triangle sequence are more worrying from a practical standpoint. Indeed, whether the demonstrated problems with $k$ edge-swaps pose a fundamental limitation not only to sampling with strict constraints, but also to importance weighted or biased samplers is worth considering.  

\section{Acknowledgments}

This work was aided by Johan Ugander's suggestion of investigating both triangle and triangle sequence constraints, and spurred by conversations with Bailey Fosdick, Dan Larremore and Johan Ugander.

 \bibliographystyle{plain} 
 \bibliography{referencesJN3}

\end{document}